\documentclass[a4paper,10pt]{article}

\usepackage{amssymb,amsmath,graphics}

\newenvironment{proof}{\par\noindent{\bf Proof.}}{\par\rightline{$\blacksquare$}}
\newtheorem{thm}{Theorem}[section]
 \newtheorem{cor}[thm]{Corollary}
 \newtheorem{lem}[thm]{Lemma}

 \newtheorem{rem}[thm]{Remark}
 \numberwithin{equation}{section}

\begin{document}

\title{\bf Bounds for stochastic processes on product index spaces}
\author{{Witold Bednorz}
\footnote{Institute of Mathematics, University of Warsaw, Banacha 2, 02-097 Warsaw, Poland}
\footnote{Research partially supported by  MNiSW Grant N N201 608740 and Mobility Plus}}
\date{\today}

\maketitle

\begin{abstract}
In this paper we discuss the question how to bound supremum of a stochastic process with the index set of a product type. There is a tempting idea to approach the question by the analysis of the process on each of the marginal index spaces separately.
However it turns out that we also need to study suitable partitions of the whole index space. 
We show what can be done in this direction and how to use the method to reprove some known results. In particular we observe that all known applications of the Bernoulli Theorem
can be obtained in this way, moreover we use the shattering dimension to slightly
extend the application to VC classes. We also show some application
to the regularity of paths for processes which take values in vector spaces.   
Finally we give a short proof of the Mendelson-Paouris result on sums of squares
for empirical processes. 
\end{abstract}

\section{Introduction}
In this paper $I$ denotes a countable set and $(F,\|\cdot\|)$ a separable Banach space.
Consider the class $\mathcal{A}$ of subsets of $I$. We say that the class $\mathcal{A}$ \textit{satisfies the maximal inequality}
if for any symmetric independent random variables $X_i$, $i\in I$ taking values in $F$ 
the following inequality holds
\begin{equation}\label{nier:1}
\mathbf{E} \sup_{A\in \mathcal{A}}\left\|\sum_{i\in A} X_i\right\|\leqslant K\mathbf{E} \left\|\sum_{i\in I}X_i\right\|,
\end{equation}
where $K$ depends on $\mathcal{A}$ only. We stress that in this paper $K$ will be used to 
denote constants that appear in the formulation of our results and may depend on 
their structure. We use $c,C,L,M$ to denote absolute constants which
may change their values from line to line by numerical factors. Also we write $\sim$
to express that two quantities are comparable up to a universal constant.
This will help us to reduce
the notation applied in this paper. It is an easy observation to see that (\ref{nier:1}) is equivalent to
\begin{equation}\label{nier:2}
\mathbf{E} \sup_{A\in \mathcal{A}}\left\|\sum_{i\in A} v_i\varepsilon_i\right\|\leqslant K\mathbf{E}\left\|\sum_{i\in I}v_i\varepsilon_i\right\|,
\end{equation}
where $(v_i)_{i\in I}$, consists of vectors in $F$ and $(\varepsilon_i)_{i\in I}$ is a Bernoulli sequence, i.e. a sequence of independent r.v.'s such that $\mathbf{P}(\varepsilon_i=\pm 1)=\frac{1}{2}$.
\smallskip

\noindent
To understand what is the right characterization 
of such classes $\mathcal{A}$ we have to recall the notion of VC dimension. We say that $\mathcal{A}$ has VC dimension $d$ if there exists a set $B\subset I$, $|B|=d$ 
such that $|\{B\cap A:\;A\in \mathcal{A}\}|=2^d$ but for all $B\subset I$, $|B|>d$, $|\{B\cap A:\;A\in \mathcal{A}\}|<2^{d+1}$.  
It means that $\mathcal{A}$ shatters some set $B$ of cardinality $d$ but do not shatter any set of cardinality $d+1$. 
The result which has been proved in \cite{Bed1} as a corollary of
the Bernoulli Theorem states
that finite VC dimension is the necessary and sufficient condition for the class $\mathcal{A}$ to have the property (\ref{nier:1}). Since our paper strongly refers to
the Bernoulli Theorem we recall its formulation. We start from mentioning Talagrand's result for 
Gaussian's processes. In order to find two-sided 
bounds for supremum of the process $G(t)=\sum_{i\in I}t_ig_i$, where $t\in T\subset \ell^2(I)$ and $(g_i)_{i\in I}$ is a Gaussian sequence, i.e. a sequence of independent standard Gaussian r.v.'s  we
need Talagrand's $\gamma_2(T)$ numbers, cf. Definition 2.2.19 in \cite{Tal1} 
or (\ref{nude}) below. By the well known Theorem 2.4.1 in \cite{Tal1} we have 
\begin{equation}\label{lasek}
\mathbf{E}\sup_{t\in T}G(t) \sim \gamma_2(T).
\end{equation}
The Bernoulli Theorem, i.e. Theorem 1.1 in \cite{Bed1}, concerns a similar question for processes of random signs.
\begin{thm}\label{tw:0}
Suppose that $T\subset \ell^2(I)$. Then
$$
\mathbf{E}\sup_{t\in T}\sum_{i\in I}t_i\varepsilon_i\sim \inf_{T\subset T_1+T_2}\left(\sup_{t\in T_1}\|t\|_1+\gamma_2(T_2)\right).
$$
where the infimum is taken over all decompositions $T_1+T_2=\{t_1+t_2:\;t_1\in T_1,t_2\in T_2\}$ that contain the set $T$ and $\|t\|_1=\sum_{i\in I}|t_i|$.
\end{thm} 
Note that if $0\in T$ then we can also require that $0\in T_2$ in the above result.
The consequence of Theorem \ref{tw:0} to our problem with maximal inequalities is as follows.
\begin{thm}\label{tw:1}
The class $\mathcal{A}$ satisfies (\ref{nier:1}) with a finite constant $K$ if and only if 
$\mathcal{A}$ is a VC class of a finite dimension. Moreover the square root of the dimension is up to a universal constant comparable with the optimal value of $K$.
\end{thm}
Observe that part of the result is obvious. Namely one can easily show that if $\mathcal{A}$ satisfies  the maximal inequality then it is necessarily
a VC class of a finite dimension. Indeed let $(\varepsilon_i)_{i\in I}$ be a Bernoulli sequence.
Suppose that set $B\subset I$ is shattered. Let $x_i=1$ for $i\in B$ and $x_i=0$, $i\not\in B$. 
Obviously
$$
\mathbf{E} \left|\sum_{i\in I}x_i \varepsilon_i\right|=\mathbf{E} \left|\sum_{i\in B}\varepsilon_i\right|\leqslant \sqrt{|B|}
$$
and on the other hand 
\begin{align*}
\mathbf{E} \sup_{A\in \mathcal{A}}\left|\sum_{i\in A}x_i \varepsilon_i\right|= \mathbf{E} \sup_{A\in \mathcal{A}}\left|\sum_{i\in A\cap B}x_i\varepsilon_i\right|\geqslant
\mathbf{E} \sum_{i\in B} \varepsilon_i 1_{\varepsilon_i=1}=|B|/2.
\end{align*}
Consequently if (\ref{nier:1}) holds then $K\geqslant \sqrt{|B|}/2$. Therefore (\ref{nier:1}) implies 
that cardinality of $B$ must be smaller or equal $4K^2$. 
\smallskip

\noindent
Much more difficult is to prove the 
converse statement, i.e that for each VC class $\mathcal{A}$ of dimension $d$ the inequality (\ref{nier:2}) holds with $K$ comparable with $\sqrt{d}$. 
In order to prove the result one has to first replace
the basic formulation of the maximal inequality - (\ref{nier:2}) by its equivalent version
\begin{equation}\label{nier:3}
\mathbf{E} \sup_{A\in \mathcal{A}}\sup_{t\in T}\left|\sum_{i\in A}t_i\varepsilon_i\right|  \leqslant K\mathbf{E} \sup_{t\in T}\left|\sum_{i\in I}t_i\varepsilon_i\right|,
\end{equation}
where $(\varepsilon_i)_{i\in I}$ is a Bernoulli sequence and $0\in T\subset \ell^2(I)$.
Note that we use absolute values since part of our work concerns 
complex spaces. However it is important to mention that in the real case
$\mathbf{E}\sup_{t\in T}\sum_{i\in I}t_i\varepsilon_i$ is comparable with $\mathbf{E}\sup_{t\in T} |\sum_{i\in I}t_i \varepsilon_i|$ if $0\in T$ and therefore we often require in this paper that $0\in T\subset \ell^2(I)$.
Let us denote
$$
b(T)=\mathbf{E} \sup_{t\in T}\left|\sum_{i\in I}t_i \varepsilon_i\right|,\;\;g(T)=\mathbf{E}\sup_{t\in T}\left|\sum_{i\in I}t_i g_i\right|,
$$ 
where $(\varepsilon_i)_{i\in I}$, $(g_i)_{i\in I}$ are respectively Bernoulli and Gaussian sequence. We recall that, what was known for a long time \cite{Kra, Lat},  (\ref{nier:3}) holds when Bernoulli random variables
are replaced by Gaussians, i.e.
\begin{equation}\label{nier:4}
\mathbf{E} \sup_{A\in \mathcal{A}}\sup_{t\in T}\left|\sum_{i\in A}g_i t_i\right|\leqslant C\sqrt{d}\mathbf{E} \sup_{t\in T}\left|\sum_{i\in I}t_i g_i\right|=C\sqrt{d}g(T),
\end{equation}
for any $0\in T\subset \ell^2(I)$. Due to Theorem \ref{tw:0} one can cover the set $T$ by $T_1+T_2$, where $0\in T_2$ and
\begin{equation}\label{nier:5}
\max\left\{\sup_{t\in T_1}\|t\|_1,g(T_2)\right\}\leqslant Lb(T).
\end{equation}
Therefore using (\ref{nier:4}) and (\ref{nier:5})
\begin{align*}
& \mathbf{E} \sup_{A\in \mathcal{A}}\sup_{t\in T}\left|\sum_{i\in A}\varepsilon_i t_i\right|\\
&\leqslant \sup_{t\in T_1}\|t\|_1+\mathbf{E} \sup_{A\in \mathcal{A}}\sup_{t\in T_2}|\sum_{i\in A}\varepsilon_i t_i| \\
&\leqslant \sup_{t\in T_1}\|t\|_1+\sqrt{\frac{\pi}{2}}\mathbf{E} \sup_{A\in \mathcal{A}}\sup_{t\in T_2}|\sum_{i\in A}t_i \varepsilon_i \mathbf{E} |g_i| | \\
&\leqslant  \sup_{t\in T_1}\|t\|_1+\sqrt{\frac{\pi}{2}}\mathbf{E} \sup_{A\in \mathcal{A}}\sup_{t\in T_2}|\sum_{i\in A} t_i g_i|  \\
& \leqslant \sup_{t\in T_1}\|t\|_1+\sqrt{\frac{\pi}{2}}g(T_2)\leqslant CL\sqrt{d}b(T).
\end{align*}
This proves Theorem \ref{tw:1}. 
\smallskip

\noindent
We show another example in which a similar approach works. Let $G$ be a compact Abelian group and $(v_i)_{i\in I}$ a sequence
of vectors taking values in $F$. Let $\chi_i$, $i\in I$ be characters on $G$. The deep result of Fernique \cite{Fer} is
$$
\mathbf{E} \sup_{h\in G}\left\|\sum_{i\in I}v_i\chi_i(h) g_i\right\|\leqslant C\left(\mathbf{E}\left\|\sum_{i\in I}v_i g_i\right\|+\sup_{\|x^{\ast}\|\leqslant 1}\mathbf{E}\sup_{h\in G}\left|\sum_{i\in I}x^{\ast}(v_i)\chi_i(h)g_i\right|\right).
$$
This can be rewritten similarly as (\ref{nier:4}), i.e. for any $0\in T\subset \ell^2(I)$ (which is a complex space in this case)
\begin{equation}\label{nier:6}
\mathbf{E} \sup_{h\in G}\sup_{t\in T}\left|\sum_{i\in I}t_i\chi_i(h)g_i\right|\leqslant C\left(g(T)+\sup_{t\in T}\mathbf{E} \sup_{h\in G}\left|\sum_{i\in I}t_i\chi_i(h)g_i\right|\right).
\end{equation}
Once again the Bernoulli Theorem allows to prove a similar result for Bernoulli sequences. Namely by Theorem \ref{tw:0}
we get the decomposition $T\subset T_1+T_2$, $0\in T_2$ such that 
\begin{equation}\label{nier:7}
\max\left\{\sup_{t\in T_1}\|t\|_1,g(T_2)\right\}\leqslant Lb(T).
\end{equation}
Consequently using (\ref{nier:6}), (\ref{nier:7}) and $|\chi_i(h)|\leqslant 1$ we get
\begin{eqnarray}
&& \mathbf{E} \sup_{h\in G}\sup_{t\in T}\left|\sum_{i\in I}t_i\chi_i(h)\varepsilon_i\right|\\
&&\leqslant  \sup_{t\in T_1}\|t\|_1+\mathbf{E} \sup_{h\in G}\sup_{t\in T_2}\left|\sum_{i\in A}\varepsilon_i t_i\chi_i(h)\right| \nonumber \\
&&\leqslant \sup_{t\in T_1}\|t\|_1+\sqrt{\frac{\pi}{2}}\mathbf{E} \sup_{h\in G}\sup_{t\in T_2}\left|\sum_{i\in I} t_i\chi_i(h) g_i\right| \nonumber\\
&&\leqslant  \sup_{t\in T_1}\|t\|_1+C\left(g(T_2)+\sup_{t\in T_2}\mathbf{E} \sup_{h\in G}\left|\sum_{i\in I} t_i\chi_i(h) g_i\right|\right) \nonumber  \\
\label{nier:8} && \leqslant CL\left(b(T)+\sup_{t\in T_2}\mathbf{E} \sup_{h\in G}|\sum_{i\in I}t_i \chi_i(h)g_i|\right).
\end{eqnarray} 
The final step is the Marcus-Pisier estimate \cite{M-P} (see Theorem 3.2.12 in \cite{Tal1})
\begin{equation}\label{nier:9}
\sup_{t\in T_2}\mathbf{E} \sup_{h\in G}\left|\sum_{i\in I}t_i \chi_i(h)g_i\right|\leqslant M\sup_{t\in T_2}\mathbf{E} \sup_{h\in G}\left|\sum_{i\in I}t_i \chi_i(h)\varepsilon_i\right|.
\end{equation}
Note that (\ref{nier:9}) is
deeply based on the translational invariance of the distance 
\begin{equation}\label{charon}
d_t(g,h)=\left(\sum_{i\in I}|t_i|^2 |\chi_i(g)-\chi_i(h)|^2\right)^{\frac{1}{2}}\;\;g,h\in G.
\end{equation}
Since we may assume that $T_2\subset T-T_1$ we get
\begin{align}
\nonumber & \sup_{t\in T_2}\mathbf{E}\sup_{h\in G}\left|\sum_{i\in I} t_i\chi_i(h)\varepsilon_i\right|\\
\nonumber & \leqslant \sup_{t\in T}\mathbf{E}\sup_{h\in G}\left|\sum_{i\in I}t_i\chi_i(h)\varepsilon_i\right|+\sup_{t\in T_1}\mathbf{E}\sup_{h\in G}\left|\sum_{i\in I}t_i\chi_i(h)\varepsilon_i\right|\\
\label{nier:9.5} &\leqslant \sup_{t\in T}\mathbf{E}\sup_{h\in G}\left|\sum_{i\in I}t_i\chi_i(h)\varepsilon_i\right|+Lb(T).
\end{align}
Combining (\ref{nier:8}) with (\ref{nier:9}) and (\ref{nier:9.5}) we get the following result.
\begin{thm}\label{tw:2}
Suppose that $0\in T\subset \ell^2(I)$.
For any compact group $G$ and a collection of vectors $v_i\in F$ in a complex Banach space $(F,\|\cdot\|)$ and characters $\chi_i$ on $G$
the following holds
$$
\mathbf{E} \sup_{h\in G}\sup_{t\in T}\left|\sum_{i\in I}t_i\chi_i(h)\varepsilon_i\right|\leqslant K\left(b(T)+\sup_{t\in T}\mathbf{E} \sup_{h\in G}\left|\sum_{i\in I}t_i \chi_i(h)\varepsilon_i\right|\right).
$$ 
\end{thm}
The aim of this note is to explore questions described above in a unified language.
Simply we consider random processes $X(u,t)$, $(u,t)\in U\times T$ with values in $\mathbb{R}$ or $\mathbb{C}$,
which means we study stochastic processes defined on product index sets. 
In particular we cover all canonical processes in this way. Indeed, suppose that 
$U \subset \mathbb{R}^{I}$ or $\mathbb{C}^{I}$ and $T\subset \mathbb{R}^{I}$ or $\mathbb{C}^{I}$ are such that for any $a\in U$ and $t\in T$ we have that $\sum_{i\in I}|u_it_i|^2<\infty$.
Then for any family of independent random variables $X_i$ such that $\mathbf{E} X_i=0$, $\mathbf{E} |X_i|^2=1$, 
$$
X(u,t)=\sum_{i\in I}u_it_iX_i,\;\;u\in U,t\in T
$$ 
is a well defined process. As we have mentioned our main classes of examples concern Gaussian canonical
processes where $X_i=g_i$, $i\in I$ are standard normal variables or Bernoulli canonical processes where $X_i=\varepsilon_i$, $i\in I$
are random signs. In particular we aim to find bounds for $\mathbf{E} \sup_{u\in U}\|\sum_{i\in I}u_i v_i \varepsilon_i\|$, where $v_i\in F$, $i\in I$, formulated in terms of $\mathbf{E}\|\sum_{i\in I}v_i \varepsilon_i\|$. One of the results we state is the application of the shattering dimension introduced by Mendelson and Vershynin \cite{M-V}, which enable us to generalize Theorem \ref{tw:1}. In this way we deduce that under mild conditions imposed on $\mathcal{A}\subset \mathbb{R}^{I}$ we have 
$$
\mathbf{E}\sup_{u\in U}\left\|\sum_{i\in I}u_iv_iX_i\right\|\leqslant K\mathbf{E}\left\|\sum_{i\in I}v_i X_i\right\|
$$ 
for any independent symmetric r.v.'s $X_i$, $i\in I$. We show how to apply the result
in the analysis of convex bodies and their volume in high dimensional spaces.  
On the other hand we can use the approach to study processes $X(t)=(X_i(t))_{i\in I}$, $t\in [0,1]$ which take values in $\mathbb{R}^{I}$ or $\mathbb{C}^{I})$.
For example to check whether paths $t\rightarrow X(t)$ belong to $\ell^2$ we should consider
$$
X(u,t)=\sum_{i\in I}u_i X_i(t),\;\;u=(u_i)_{i\in I}\in U,t\in[0,1],
$$
where $U$ is the unit ball in $\ell^2(I)$, i.e. $U=\{u\in\mathbb{R}^{I}:\;\sum_{i\in I}|u_i|^2\leqslant 1\}$.
The finiteness of $\|X(t)\|_2<\infty$ is equivalent to the finiteness of $\sup_{u\in u}|X(u,t)|$.
Similarly we can treat a well known question in the theory of empirical processes.
Suppose that $(\mathcal{E},\mathcal{B})$ is a measurable space and $\mathcal{F}$ a countable family of measurable real functions on $\mathcal{E}$.
Let $X_1,X_2,\ldots,X_N$ be independent random
variables which take values in $(\mathcal{E},\mathcal{B})$, we may define
$$
X(u,f)=\sum^N_{i=1}u_if(X_i),\;\;u=(u_i)^N_{i=1}\in U,f\in \mathcal{F},
$$ 
where $U=B_N(0,1)=\{u\in \mathbb{R}^N:\;\sum^N_{i=1}|u_i|^2\leqslant 1\}$.
Then it is clear that
$$
\sup_{u\in U}|X(u,f)|^2=\sum^N_{i=1}|f(X_i)|^2,\;\;\mbox{for all}\; f\in \mathcal{F}.
$$
In the last section we give a short proof of Mendelson-Paouris result \cite{Me-P} that
gives an upper bound for $\mathbf{E}\sup_{u\in U}\sup_{f\in \mathcal{F}}|X(u,t)|$.

\section{Upper bounds}

For the sake of completeness we give an idea how to bound stochastic processes. 
The approach we present slightly extends former results of Latala \cite{Lat2,Lat1} and Mendelson-Paouris \cite{Me-P}.
Suppose that $\mathbf{E} |X(t)-X(s)|<\infty$ for all $s,t\in T$.
For each $s,t\in T$ and $n\geqslant 0$ we define $\bar{q}_n(s,t)$ as the smallest $q>0$ such that
\begin{align}
\nonumber & F_{q,n}(s,t)=\mathbf{E} q^{-1}(|X(t)-X(s)|-q)_{+}\\
\label{holy}&=\int^{\infty}_{1}\mathbf{P}(|X(t)-X(s)|>qt)dt\leqslant N_n^{-1}.
\end{align}
We prove the following observation.
\begin{lem}
Function $\bar{q}_n(s,t)$, $s,t\in T$ is a distance on $T$, namely is symmetric, satisfies the triangle inequality and $\bar{q}_n(s,t)=0$ if and only if $X(s)=X(t)$ a.s. 
\end{lem}
\begin{proof}
Obviously $\bar{q}_n(s,t)$ is finite and symmetric $\bar{q}_n(s,t)=\bar{q}_n(t,s)$. To see that it equals $0$ if and only if $\mathbf{P}(|X(t)-X(s)|>0)>0$ a.s. note that if $X(s)\neq X(t)$ then $\mathbf{E}|X(t)-X(s)|>0$.
The function $q\rightarrow F_{q,n}(s,t)$ is decreasing continuous
and $F_{q,n}(s,t)\rightarrow \infty$ if $q\rightarrow 0$ and $\mathbf{E}|X(t)-X(s)|>0$. Moreover $F_{q,n}(s,t)$ is strictly decreasing on the interval $\{q>0:\;F_{q,n}(s,t)>0\}$ and consequently $\bar{q}(s,t)$ is the unique solution of $F_{q,n}(s,t)=N_n^{-1}$, namely
$$
\mathbf{E} (\bar{q}_n(s,t))^{-1}(|X(t)-X(s)|-\bar{q}_n(s,t))_{+}=N_n^{-1}.
$$
Finally we show that $\bar{q}$ satisfies the triangle inequality. Indeed for any $u,v,w\in T$ either
$\bar{q}(u,v)=0$ or $\bar{q}(v,w)=0$ or $\bar{q}(u,w)=0$ and the inequality is trivial
or all the quantities are positive and then 
$$
F_{\bar{q}_n(u,v),n}(u,v)=F_{\bar{q}_n(v,w),n}(v,w)=F_{\bar{q}_n(u,w),n}(u,w)=N_{n}^{-1}.
$$ 
It suffices to observe
\begin{align*}
& \frac{1}{\bar{q}_n(u,v)+\bar{q}_n(v,w)}\mathbf{E}\left(|X(u)-X(w)|-\bar{q}_n(u,v)-\bar{q}_n(w,v)\right)_{+} \\
&\leqslant \mathbf{E}\left(\frac{|X(u)-X(w)|+|X(w)-X(v)|}{\bar{q}_n(u,v)+\bar{q}_n(w,v)}-1\right)_{+}.
\end{align*}
The function $x\rightarrow (x-1)_{+}$ is convex which implies that 
\begin{equation}\label{gniew}
(px+qy-1)_{+}\leqslant p(x-1)_{+}+q(y-1)_{+}
\end{equation}
for $p,q\geqslant 0$, $p+q=1$ and $x,y>0$.
We use (\ref{gniew}) for
$$
x=\frac{|X(u)-X(v)|}{\bar{q}_n(u,v)},\;\;y=\frac{|X(v)-X(w)|}{\bar{q}_n(v,w)}
$$
and
$$
p=\frac{\bar{q}_n(u,v)}{\bar{q}_n(u,v)+\bar{q}_n(w,v)},\;\;q=\frac{\bar{q}_n(v,w)}{\bar{q}_n(u,v)+\bar{q}_n(w,v)}.
$$
Therefore
\begin{align*}
& \frac{1}{\bar{q}_n(u,v)+\bar{q}_n(v,w)}\mathbf{E}\left(|X(u)-X(w)|-\bar{q}_n(u,v)-\bar{q}_n(w,v)\right)_{+} \\
&\leqslant p\mathbf{E}\left(\frac{|X(u)-X(v)|}{\bar{q}_n(u,v)}-1\right)_{+}+q\mathbf{E}\left(\frac{|X(v)-X(w)|}{\bar{q}_n(v,w)}-1\right)_{+} \\
&\leqslant pN_n^{-1}+qN_{n}^{-1}=N_n^{-1}
\end{align*}
which by definition gives that
$$
\bar{q}_n(u,v)+\bar{q}_n(v,w)\geqslant \bar{q}_n(v,w).
$$
\end{proof}
Obviously usually we do not need to use the optimal distances $\bar{q}_n$
and replace the construction which is sufficient for our purposes. We say that a family of distances $q_n$, $n\geqslant 0$ on $T$ is \textit{admissible} if  
$q_n(s,t)\geqslant \bar{q}_n(s,t)$ and $q_{n+1}(s,t)\geqslant q_n(s,t)$.
For example Latala \cite{Lat2,Lat1} and Mendelson-Paouris \cite{Me-P} used moments, namely if
$\|X(t)-X(s)\|_p=(\mathbf{E}|X(t)-X(s)|^p)^{\frac{1}{p}}<\infty$, $p\geqslant 1$ then we may take $q_n(s,t)=2\|X(t)-X(s)\|_{2^n}$.
Indeed observe that
\begin{align*}
&\mathbf{E} \left(\frac{|X(t)-X(s)|}{2\|X(t)-X(s)\|_{2^n}}-1\right)_{+}\leqslant \frac{\mathbf{E} |X(t)-X(s)|^{2^n}}{\left(2\|X(t)-X(s)\|_{2^n}\right)^{2^n}}\leqslant \frac{1}{N_n}.
\end{align*}
Following Talagrand we say that a sequence of partitions $\mathcal{A}_n$, $n\geqslant 0$ of a set $T$ is \textit{admissible} if it is increasing, $\mathcal{A}_0=\{T\}$ and $|\mathcal{A}_n|\leqslant N_n$. 
Let us also define admissible sequences of partitions $\mathcal{A}=(\mathcal{A}_n)_{n\geqslant 0}$ of the set $T$, which means nested sequences of partitions such that $\mathcal{A}_0=\{T\}$
and $|\mathcal{A}_n|\leqslant N_n$. For each $A\in \mathcal{A}_n$ we define
$$
\Delta_n(A)=\sup_{s,t\in A}q_n(s,t).
$$
By $A_n(t)$ we call devote the element of $\mathcal{A}_n$ that contains point $t$. For each $A\in \mathcal{A}_n$, $n\geqslant 0$ we define $t_A$ as an arbitrary point in $A$.
We may and will assume that if $t_A\in  B\in \mathcal{A}_{n+1}$, $B\subset A\in \mathcal{A}_n$ then $t_B=t_A$.
Let $\pi_n(t)=t_{A_n(t)}$, then $\pi_0(t)=t_{T}$ is a fixed point in $T$. Let $T_n=\{\pi_n(t):\;t\in T\}$ for $n\geqslant 0$. Clearly $T_n$, $n\geqslant 0$ are nested, namely $T_n\subset T_{n+1}$ for $n\geqslant 0$. For each stochastic process $X(t)$, $t\in T$ and $\tau\geqslant 0$ we may define 
$$
\gamma^{\tau}_X(T)=\inf_{\mathcal{A}}\sup_{t\in T}\sum^{\infty}_{n=0}\Delta_{n+\tau}(A_n(t)).
$$ 
We prove that for $\tau\geqslant 2$, $\gamma^{\tau}_X(T)$ is a good upper bound for $\mathbf{E} \sup_{s,t\in T}|X(t)-X(s)|$.
\begin{thm}\label{tw:3}
For $\tau\geqslant 2$ the following inequality holds
$$
\mathbf{E} \sup_{s,t\in T}|X(t)-X(s)|\leqslant 4\gamma^{\tau}_X(T).
$$
\end{thm}
\begin{proof}
Note that
\begin{align*}
&|X(t)-X(\pi_0(t))|\\
&\leqslant \sum^{\infty}_{n=0} q_{n+\tau}(\pi_{n+1}(t),\pi_{n}(t))\\
&+\sum^{\infty}_{n=0}\left(|X(\pi_{n+1}(t))-X(\pi_n(t))|-q_{n+\tau}(\pi_{n+1}(t),\pi_n(t))\right)_{+} \\
&\leqslant \sum^{\infty}_{n=0}\Delta_{n+\tau}(A_n(t))+\sum^{\infty}_{n=0}\sum_{u\in T_{n}}\sum_{v\in A_{n}(u)\cap T_{n+1}}\left(|X(u)-X(v)|-q_{n+\tau}(u,v)\right)_{+}.
\end{align*}
For any $\varepsilon>0$ one can find nearly optimal admissible partition 
$(\mathcal{A}_n)_{n\geqslant 0}$ such that 
$$
\sup_{t\in T}\sum^{\infty}_{n=0}\Delta_{n+\tau}(A_n(t))\leqslant (1+\varepsilon)\gamma^{\tau}_X(T)
$$
and therefore
\begin{align*}
&\mathbf{E} \sup_{t\in T}|X(t)-X(\pi_0(t))|\\
&\leqslant \gamma^{\tau}_X(T)+\varepsilon+\sum^{\infty}_{n=0}\sum_{u\in T_n}\sum_{v\in A_{n}(u)\cap T_{n+1}}\frac{q_{n+\tau}(u,v)}{N_{n+\tau}} \\
&\leqslant (1+\varepsilon)\gamma^{\tau}_X(T)+\sum^{\infty}_{n=0}\sum_{u\in T_n}\Delta_{n+\tau}(A_n(u))\frac{N_{n+1}}{N_{n+\tau}} \\
&\leqslant  \gamma^{\tau}_X(T)\left(1+\varepsilon+\sum^{\infty}_{n=0}\frac{N_n N_{n+1}}{N_{n+\tau}}\right) \\
&\leqslant  (1+\varepsilon)\gamma^{\tau}_X(T)+
\gamma^{\tau}_X(T)\left(\frac{N_0N_1}{N_2}+\sum^{\infty}_{n=1}\frac{1}{N_n}\right)\\
&\leqslant \gamma^{\tau}_X(T)\left(1+\varepsilon+\sum^{\infty}_{n=0}2^{-n-1}\right),
\end{align*}
where in the last line we used $N_n^2=N_{n+1}$,
$N_n\geqslant {2^{n+1}}$ for $n\geqslant 1$ and $N_0N_1/ N_2=1/4< 1/2$.
Since $\varepsilon$ is arbitrary small we infer that
$\mathbf{E}\sup_{t\in T}|X(t)-X(\pi_0(t))|\leqslant 2\gamma^{\tau}_X(T)$ and hence
$\mathbf{E} \sup_{s,t}|X(t)-X(s)|\leqslant 4\gamma_X^{\tau}(T)$.
\end{proof}
The basic question is how to construct  admissible sequences of partitions. The simplest way to do this goes through entropy numbers. Recall that for a given distance $\rho$ on $T$
the quantity $N(T,\rho,\varepsilon)$ denotes the smallest number of balls of radius $\varepsilon$ with respect to $\rho$ necessary to cover the set $T$.
Consequently for a given $\tau\geqslant 0$ we define entropy numbers as
$$
e^{\tau}_n=\inf\{\varepsilon>0:\;N(T,q_{n+\tau},\varepsilon)\leqslant N_{n}\},\;\;n\geqslant 0.
$$
Having the construction ready we may easily produce an admissible sequence of partitions. On each level $n$ there exists 
at most $N_{n}$ sets of $q_{n+\tau}$ diameter $2\varepsilon$ that covers $T$. To obtain nested sequence of partitions we have to intersect all the sets constructed
at levels $0,1,\ldots,n-1$. The partition $\mathcal{A}_n$ has no more than $N_0N_1\ldots N_{n-1}\leqslant N_n$ elements. Moreover 
for each set $A\in \mathcal{A}_n$ we have 
$$
\Delta_{n+\tau-1}(A)\leqslant 2e^{\tau}_{n-1}\;\;\mbox{for}\;n\geqslant 1.
$$ 
Obviously $\Delta_{\tau-1}(T)\leqslant \Delta_{\tau}(T)\leqslant 2e^{\tau}_0$.
Let $\mathcal{E}^{\tau}_X(T)=\sum^{\infty}_{n=0}e^{\tau}_n$, then for any $\tau\geqslant 1$ 
$$
\gamma^{\tau-1}_{X}(T)\leqslant 2e^{\tau}_0+2\sum^{\infty}_{n=1}e^{\tau}_{n-1}\leqslant 4\mathcal{E}^{\tau}_X(T).
$$
and hence by Theorem \ref{tw:3} with $\tau\geqslant 3$
\begin{equation}\label{nier2}
\mathbf{E}\sup_{s,t\in T}|X(t)-X(s)|\leqslant 16\mathcal{E}^{\tau}_X(T).
\end{equation}
We turn to our main question of processes on product spaces. 
Consider $X(u,t)$, $u\in U$, $t\in T$ we define two different families of distances
$q_{n,t}$ and $q_{n,u}$ admissible to control marginal processes respectively $u\rightarrow X(u,t)$ on $U$ and 
$t\rightarrow X(u,t)$ on $T$.
\smallskip

\noindent
First for a given $t\in T$, let
$q_{n,t}$, $n\geqslant 0$ be a family of distances on $U$  admissible for the process $u\rightarrow X(u,t)$ and let $e^{\tau}_{n,t}$, $n\geqslant 0$ be
entropy numbers on $U$ constructed for $q_{n,t}$, $n\geqslant 0$. By (\ref{nier2}) we infer that for $\tau\geqslant 3$  
\begin{equation}\label{nier3}
\mathbf{E}\sup_{u,v\in U}|X(u,t)-X(v,t)|\leqslant 16\sum^{\infty}_{n=0}e^{\tau}_{n,t}.
\end{equation}
If we define $\mathcal{E}^{\tau}_{X,t}(U)=\sum^{\infty}_{n=0}e^{\tau}_{n,t}$ and $\mathcal{E}^{\tau}_{X,T}(U)=\sup_{t\in T}\mathcal{E}^{\tau}_{X,t}(U)$ then we may rewrite (\ref{nier3}) as
$$
\sup_{t\in T}\mathbf{E} \sup_{u,v\in U}|X(u,t)-X(v,t)|\leqslant 16\mathcal{E}^{\tau}_{X,T}(U),
$$
and in this way the entropy numbers may be used to bound the family of processes $u\rightarrow X(u,t)$ where $t\in T$.
\smallskip

\noindent
On the other hand for a given $u\in U$ let $q_{n,u}$, $n\geqslant 0$ be a family of distances on $T$ 
admissible for $t\rightarrow X(u,t)$. Obviously $q_{n,U}=\sup_{u\in U}q_{n,u}$ is a good upper bound for all distances $q_{n,u}$.
Let 
$$
\gamma^{\tau}_{X,U}(T)=\inf_{\mathcal{A}}\sup_{t\in T}\sum^{\infty}_{n=0}\Delta_{n+\tau,U}(A_n(t))
$$
where the infimum is taken over all admissible partitions $\mathcal{A}=(\mathcal{A}_n)_{n\geqslant 0}$ of $T$ and 
$$
\Delta_{n,U}(A)=\sup_{s,t\in A}q_{n,U}(s,t)=\sup_{s,t\in A}\sup_{u\in U}q_{n,u}(s,t).
$$  
Theorem \ref{tw:3} applied to distances $q_{n,U}$, $n\geqslant 0$ implies that
$$
\sup_{u\in U}\mathbf{E}\sup_{s,t\in T}|X(u,t)-X(u,s)|\leqslant 4\gamma^{\tau}_{X,U}(T).
$$
We prove that the described two quantities i.e. $\mathcal{E}^{\tau}_{X,T}(U)$ and $\gamma^{\tau}_{X,U}(T)$ suffice to control the process $X(u,t)$, $u\in U$, $t\in T$.
\smallskip

\noindent
We state our main result which extends the idea described in Theorem 3.3.1 in \cite{Tal1}.
\begin{thm}\label{tw:4}
For any $\tau\geqslant 4$ the following inequality holds 
$$
\mathbf{E} \sup_{u,v\in U}\sup_{s,t\in T} |X(u,t)-X(v,s)|\leqslant 24(\gamma^{\tau}_{X,U}(T)+\mathcal{E}^{\tau}_{X,T}(U)).
$$
\end{thm}
\begin{proof}
We first observe that if $\gamma^{\tau}_{X,U}(T)<\infty$ then, by
the definition, for any $\varepsilon>0$ there exists an admissible partition sequence $\mathcal{C}=(\mathcal{C}_n)_{n\geqslant 0}$ of $T$, which satisfies
$$
(1+\varepsilon)\gamma^{\tau}_{X,U}(T)\geqslant \sup_{t\in T}\sum^{\infty}_{n=0}\Delta_{n+\tau,U}(B_n(t)).
$$
Let us fix $C\in\mathcal{C}_n$ and let $\pi_n(C)$ be a point in $T$ such that 
$$
e^{\tau}_{n,\pi_n(C)}\leqslant (1+\varepsilon)\inf\{e^{\tau}_{n,t}:\;t\in B\}.
$$
Consequently, for any $n\geqslant 2$ there exists a partition $\mathcal{B}_{C,n-2}$ of the set $U$ into at most $N_{n-2}$ sets $B$ that satisfy
$$
\Delta_{n+\tau-2,\pi_{n-2}(C)}(B)\leqslant 2e^{\tau}_{n-2,\pi_{n-2}(C)}, \;\;\mbox{where}\; \Delta_{n,t}(B)=\sup_{u,v\in B}q_{n,t}(u,v).
$$
Using sets $B\times C$ for $B\in \mathcal{B}_{C,n-2}$ and $C\in \mathcal{C}_{n-2}$
we get a partition $\mathcal{A}_{n-2}'$ of $U\times T$ into at most $N_{n-2}^2\leqslant N_{n-1}$ sets. Finally
intersecting all the constructed sets in $\mathcal{A}_0',\mathcal{A}_1',\ldots,\mathcal{A}_{n-2}'$ we obtain a nested sequence of partitions $(\mathcal{A}_n)_{n\geqslant 2}$ such that $|\mathcal{A}_n|\leqslant N_{n}$.
We complete the sequence by $\mathcal{C}_0=\mathcal{C}_1=\{U\times T\}$.
In this way $\mathcal{A}=(\mathcal{A}_{n})_{n\geqslant 0}$ is an admissible sequence of partitions for $U\times T$. 
Let $A_n(u,t)$ be the element of $\mathcal{A}_n$ that contains point $(u,t)$. Clearly
$$
A_n(u,t)\subset B\times C,  
$$
where $C=C_{n-2}(t)$ and $u\in B\in \mathcal{B}_{C,n-2}$. Therefore for $n\geqslant 2$,
$$
\sup_{s,s'\in C}q_{n+\tau-2,U}(s,s')\leqslant \Delta_{n+\tau-2,U}(C)
$$
and 
$$
\sup_{v,v'\in B}q_{n+\tau-2,\pi_n(C)}(v,v')\leqslant 2e^{\tau}_{n-2,\pi_{n-2}(C)}\leqslant 2(1+\varepsilon)e^{\tau}_{n-2,t}.
$$
We turn to the analysis of optimal quantiles $\bar{q}_n$ for the process $X(u,t)$, $u\in\mathcal{A}$, $t\in T$. We show that for any $x,y,z\in T$ and $v,w\in U$
$$
\bar{q}_{n}((v,x),(w,y))\leqslant q_{n,U}(x,z)+q_{n,U}(y,z)+q_{n,z}(v,w).
$$
This holds due to the triangle inequality
\begin{align*}
&\bar{q}_{n}((v,x),(w,y))\\
&\leqslant  \bar{q}_{n}((v,x),(v,z))+\bar{q}_{n}((w,y),(w,z))+
\bar{q}_{n}((v,z),(w,z))\\
&\leqslant  q_{n,v}(x,z)+q_{n,w}(y,z)+q_{n,z}(v,w)\\
&\leqslant  q_{n,U}(x,z)+q_{n,U}(y,z)+q_{n,z}(v,w).
\end{align*}
In particular it implies that for any $(v,s)\in B\times C$
\begin{align*}
&\bar{q}_{n+\tau-2}((u,t),(v,s))\\
&\leqslant q_{n+\tau-2,U}(t,\pi_{n-2}(C))+q_{n+\tau-2,U}(s,\pi_{n-2}(C))+q_{n+\tau-2,\pi_{n-2}(C)}(u,v)
\end{align*}
and hence
$$
\Delta_{n+\tau-2}(B\times C)\leqslant 2\Delta_{n+\tau-2,U}(C_{n-2}(t))+2(1+\varepsilon)e^{\tau}_{n-2,t}.
$$
If $\tau\geqslant 2$ then also for $n=0,1$,
$$
\Delta_{n+\tau-2}(U\times T)\leqslant \Delta_{\tau}(U\times T)\leqslant 2\Delta_{\tau}(T)+2(1+\varepsilon)\sup_{t\in T}e^{\tau}_{0,t}.
$$ 
It implies that for any $\tau\geqslant 2$
$$
\gamma^{\tau-2}_{X}(U\times T)\leqslant 4\Delta_{\tau}(T)+4(1+\varepsilon)\sup_{t\in T}e^{\tau}_{0,t}+2(1+\varepsilon)\gamma^{\tau}_{X,U}(T)
+2(1+\varepsilon)\mathcal{E}^{\tau}_{X,T}(U).
$$
Therefore for any $\tau\geqslant 4$ we may apply Theorem \ref{tw:2} for distances $\bar{q}_n$  and in this way prove our result with the constant $24$.
\end{proof}
In the next sections we analyse various applications of the result.

\section{Gaussian case}

As we have mentioned the basic example for the theory are Gaussian canonical processes.
Suppose that $T\subset \ell^2(I)$, we recall that
$$
G(t)=\sum_{i\in I}t_ig_i,\;\;t\in T,
$$
where $(g_i)_{i\in I}$ is a sequence of i.i.d. standard normal variables.
For the process $G(t)$, $t\in T$ the natural distance is 
$$
d(s,t)=(\mathbf{E}|G(t)-G(s)|^2)^{\frac{1}{2}}=\|t-s\|_2,\;\;s,t\in T.
$$
It is easy to see that the optimal quantiles for $G$ satisfy
$$
\bar{q}_n(s,t)\sim 2^{\frac{n}{2}}d(s,t),\;\;\mbox{for all}\;s,t\in T.
$$
Consequently denoting $q_n(s,t)=C2^{\frac{n}{2}}d(s,t)$ for large enough $C$ we get an admissible
sequence of distances. Moreover,
\begin{equation}\label{nude}
\gamma_G(T)\sim \gamma_2(T,d)=\inf_{\mathcal{A}}\sup_{t\in T}\sum^{\infty}_{n=0}2^{\frac{n}{2}}\Delta(A_n(t)),
\end{equation}
where the infimum is taken over all admissible $\mathcal{A}=(\mathcal{A}_n)_{n\geqslant 0}$ sequences of partitions and $\Delta(A)=\sup_{s,t\in A}d(s,t)$.
Since $d$ is the canonical distance for $\ell^2(I)$ we usually suppress $d$ from $\gamma_2(T,d)$ and simply write $\gamma_2(T)$.
As we have mentioned in the introduction, using (\ref{lasek}) we get
$$
K^{-1}\gamma_2(T)\leqslant \mathbf{E} \sup_{s,t\in T}|G(t)-G(s)|\leqslant K\gamma_2(T).
$$
Let us also define
$$
e_n=\inf\{\varepsilon:\;N(T,d,\varepsilon)\leqslant N_n\},\;\;n\geqslant 0.
$$
Obviously $e^{\tau}_n\leqslant C2^{\frac{n+\tau}{2}}e_n$ and hence
$$
\mathcal{E}^{\tau}_{G}(T)\leqslant C2^{\frac{\tau}{2}}\sum_{n\geqslant 0}2^{\frac{n}{2}}e_n.
$$
Let $\mathcal{E}(T,d)=\sum^{\infty}_{n= 0}2^{\frac{n}{2}}e_n$, it is the well known fact 
(Theorem 3.1.1 in \cite{Tal1}) that
\begin{equation}\label{min2}
K^{-1}\mathcal{E}(T,d)\leqslant \int^{\infty}_0 \sqrt{\log(N(T,d,\varepsilon))}d\varepsilon \leqslant K\mathcal{E}(T,d).
\end{equation}
Again since $d$ is the canonical distance on $\ell^2(I)$ we will suppress $d$ from $\mathcal{E}(T,d)$ and write $\mathcal{E}(T)$.
By (\ref{nier2}) we infer Dudley's bound
$$
\mathbf{E}\sup_{s,t\in T}|G(t)-G(s)|\leqslant 16C2^{\frac{\tau}{2}}\mathcal{E}(T).
$$
Turning our attention to product spaces let us 
recall that for $U\subset \mathbb{R}^{I}$ or $\mathbb{C}^{I}$, $T\subset \mathbb{R}^{I}$ or $C^I$ such that
$ut=(u_it_i)_{i\in I}\in \ell^2(I)$ for all $u\in U$ and $t\in T$ we may define
$$
G(u,t)=\sum_{i\in I}u_i t_i g_i,\;\;u\in U,t\in T, 
$$
where $(g_i)_{i\in I}$ is a Gaussian sequence.
Note that for all $s,t\in T$, $u,v\in U$
$$
\bar{q}_n((u,t),(v,s))\leqslant C2^{\frac{n}{2}}\|ut-vs\|_2.
$$
For each $u\in U$ and $s,t\in T$ let $d_u(s,t)=\|u(t-s)\|_2$.
We may define 
$$
q_{n,u}(s,t)=C2^{\frac{n}{2}}d_u(s,t),\;\;
q_{n,U}(s,t)=C2^{\frac{n}{2}}\sup_{u\in U}d_u(s,t).
$$
In particular
$$
q_{n,U}(s,t)\leqslant C\sup_{u\in U}\|u\|_{\infty}d(s,t)
$$
and therefore
\begin{equation}\label{min3}
\gamma^{\tau}_{G,U}(T)\leqslant 
C2^{\frac{\tau}{2}}\sup_{u\in U}\|u\|_{\infty}\gamma_2(T,d).
\end{equation}
On the other hand we define for all $t\in T$ and $u,v\in U$ 
\begin{equation}\label{nerf}
q_{n,t}(u,v)=C2^{\frac{n}{2}}\|t(u-v)\|_2.
\end{equation}
For each $t\in T$ let us denote by $d_t$ the distance on $U$ given by 
$$
d_t(u,v)=\left(\sum_{i\in I} |t_i|^2|u_i-v_i|^2\right)^{\frac{1}{2}},\;\;u,v\in U.
$$ 
Using these distances we may rewrite (\ref{nerf}) as $q_{n,t}= C2^{\frac{n}{2}}d_t$.
Let
$$
e_{n,t}=\inf\{\varepsilon:\;N(U,d_t,\varepsilon)\leqslant N_n\},\;\;n\geqslant 0\;\;\mbox{and}\;\;\mathcal{E}(U,d_t)=\sum^{\infty}_{n=0}2^{\frac{n}{2}}e_{n,t},
$$
we obtain by (\ref{nerf})
$$
e^{\tau}_{n,t}\leqslant C2^{\frac{n+\tau}{2}}e_{n,t}\;\;\mbox{and}\;\;
\mathcal{E}^{\tau}_{X,T}(U)\leqslant C2^{\frac{\tau}{2}}\sup_{t\in T}\mathcal{E}(U,d_t).
$$
Using (\ref{min2}) we have 
$$
K^{-1}\mathcal{E}(U,d_t)\leqslant \int^{\infty}_0 \sqrt{\log(N(U,d_t,\varepsilon))}d\varepsilon \leqslant K\mathcal{E}(U,d_t).
$$
We recall also that if $0\in T$ then by (\ref{lasek}) $\gamma_2(T)\sim g(T)=\mathbf{E}\sup_{t\in T}|G(t)|$.
We may state the following corollary of Theorem \ref{tw:4} which extends slightly Theorem 3.3.1 in \cite{Tal1}.
\begin{cor}\label{cor:1}
For any $\tau\geqslant 4$
\begin{align*}
&\mathbf{E} \sup_{u,v\in U}\sup_{s,t\in T} |G(u,t)-G(v,s)|\\
&\leqslant 32(\gamma^{\tau}_{G,U}(T)+\mathcal{E}^{\tau}_{G,T}(U)) \leqslant
32C2^{\frac{\tau}{2}} \left(\sup_{u\in U}\|u\|_{\infty}\gamma_2(T)+\sup_{t\in T} \mathcal{E}(U,d_t)\right).
\end{align*}
Moreover 
$$
\mathcal{E}(U,d_t)\sim \int^{\infty}_0 \sqrt{\log N(U,d_t,\varepsilon)}d\varepsilon\;\;\mbox{and}\;\;\gamma_2(T)\sim g(T)\;\;\mbox{if}\;0\in T.
$$
\end{cor}
It is a tempting idea to replace $\sup_{u\in U}\|u\|_{\infty}\gamma_2(T)$ and $\sup_{t\in T}\mathcal{E}(U,d_t)$
by respectively $\sup_{u\in U}\gamma_2(T,d_u)$ and $\sup_{t\in T}\gamma_2(U,d_t)$.
We show that this approach cannot work, to this aim lets us consider toy example where $T$ and $U$ are
usual ellipsoids, i.e.
\begin{equation}\label{mordka}
U=\left\{u\in \mathbb{R}^I:\;\sum_{i\in I}\frac{|u_i|^2}{|x_i|^2}\leqslant 1\right\}.
\end{equation}
and
$$
T=\left\{t\in \mathbb{R}^I:\;\sum_{i\in I}\frac{|t_i|^2}{|y_i|^2}\leqslant 1\right\}.
$$
Obviously
$$
\mathbf{E}  \sup_{u\in U,t\in T}|G(u,t)|=\mathbf{E} \max_{i\in I}|x_iy_ig_i|
$$
on the other hand 
$$
\sup_{u\in U} \mathbf{E}\sup_{t\in T}|G(u,t)|\sim \sup_{u\in U}\|u y\|_2=\max_{i\in I}|x_iy_i|=\|xy\|_{\infty}.
$$
Similarly, $\sup_{t\in T}\mathbf{E} \sup_{u\in U}|G(u,t)|\sim \|xy\|_{\infty}$. However $\|xy\|_{\infty}\leqslant 1$
does not guarantee that $\mathbf{E} \max_{i\in I}|x_iy_ig_i|$ is finite, for example if $x_i=y_i=1$.
On the other hand Corollary \ref{cor:1} implies the following result for the ellipsoid $U$. 
\begin{rem}\label{rem:1}
Suppose that $\mathcal{A}$ is given by (\ref{mordka}). Then for any set $0\in T\subset \ell^2(I)$,
$$
\mathbf{E} \sup_{u\in U}\sup_{t\in T}|G(u,t)|\leqslant K(\|x\|_{\infty}g(T)+\Delta(T)\|x\|_2).
$$
\end{rem}
\begin{proof} For the sake of simplicity we assume that $I=\mathbb{N}$.
We use Corollary \ref{cor:1} together with the following observation. For each $t\in T$ points $(u_i t_i)_{i\in \mathbb{N}}$, $u\in U$
forms an ellipsoid $U_t=\{a\in \mathbb{R}^{\mathbb{N}}:\;\sum_{i\in \mathbb{N}}\frac{|u_i|^2}{|t_i|^2|x_i|^2}\leqslant 1\}$ and therefore by Proposition 2.5.2 in \cite{Tal1}
$$
\int^{\infty}_0 \sqrt{\log N(U,d_t,\varepsilon)}d\varepsilon \leqslant L\sum^{\infty}_{n=0} 2^{\frac{n}{2}}|x_{2^n}t_{2^n}|,
$$
where $L$ is a universal constant, assuming that we have rearranged $(x_i t_i)_{i\in \mathbb{N}}$ in such a way that the sequence $|x_i t_i|$, $i\in \mathbb{N}$ is non-increasing. 
It suffices to note that by the Schwartz inequality
\begin{align*}
&\sum^{\infty}_{n=0} 2^{\frac{n}{2}}(|x_{2^n}t_{2^n}|^{2})\leqslant 2|x_1t_1|+2\sum^{\infty}_{n=1} \left(\sum^{2^{n}}_{i=2^{n-1}+1}|x_{i}t_i|^2\right)^{\frac{1}{2}} \\
&\leqslant 2|x_1t_1|+2\sum^{\infty}_{n=1}\max_{2^{n-1}<i\leqslant 2^n}|x_i|\left(\sum^{2^n}_{i=2^{n-1}+1}|t_i|^2\right)^{\frac{1}{2}}\leqslant 2\|x\|_2\|t\|_2.
\end{align*}
Consequently, 
$$
\mathbf{E} \sup_{u\in U}\sup_{t\in T}|G_{t,u}|\leqslant K(\|x\|_{\infty}\gamma_2(T)+\Delta(T)\|x\|_2).
$$
It remains to apply (\ref{lasek}), i.e. $g(T)\sim \gamma_2(T)$.
\end{proof}
Note that $\|x\|_{\infty}=\sup_{u\in U}\|u\|_{\infty}$ and $\|x\|_2\sim \gamma_2(U,d)$.
It is clear the result can be slightly improved if $\mathcal{E}(U,d)<\infty$. Indeed similarly to (\ref{min3})
one can show that
$$
\mathcal{E}^{\tau}_{G,U}(T)\leqslant C2^{\frac{\tau}{2}}\sup_{t\in T}\|t\|_{\infty}\mathcal{E}(U,d)
$$
and hence by Corollary \ref{cor:1} for any set $U$ such that $\mathcal{E}(U,d)<\infty$
$$
\mathbf{E} \sup_{u\in U}\sup_{t\in T}|G(u,t)|\leqslant K\left(\|x\|_{\infty}g(T)+\sup_{t\in T}\|t\|_{\infty}\right).
$$
This is a sort of general rule that we expect certain regularity of $\mathcal{A}$, whereas
set $T$ is usually supposed to be unknown. 
\smallskip 

\noindent
We turn to show that Corollary \ref{cor:1} is sufficiently strong to answer both questions 
about VC classes and Fourier series, posed in the introduction, in the Gaussian type formulation.
Namely we  prove that these problems are related to certain properties of entropy functionals 
$\int^{\infty}_0 \sqrt{\log(N(\mathcal{A},d_t,\varepsilon))}d\varepsilon$, $t\in T$.
\smallskip

\noindent
We start from the result on VC classes of finite dimension. In this case $\mathcal{A}$ consists of $a\in \mathbb{R}^I$, of the form $a=1_A$ for some $A\subset I$
and $T$ is any subset of $\ell^2(I)$ that contains $0$.
By Theorem 14.12 in \cite{L-T} we have that for any $t\in \ell_2(I)$
such that $\|t\|_2=1$ and given VC class $\mathcal{A}$ of dimension $d$
$$
\log N(\mathcal{A},d_t,\varepsilon)\leqslant Ld\left(1+\log \frac{1}{\varepsilon}\right),\;\;0<\varepsilon<1.
$$ 
Consequently
$$
\sup_{t\in T}\int^{\infty}_0 \sqrt{\log N(\mathcal{A},d_t,\varepsilon)}d\varepsilon\leqslant \sqrt{Ld}\Delta(T)\int^1_0 (1+\log \frac{1}{\varepsilon})^{\frac{1}{2}}d\varepsilon\leqslant  M\sqrt{d}\Delta(T) 
$$
and hence $\mathcal{E}(\mathcal{A},d_t)\leqslant M\sqrt{d}\Delta(T)$ for any $t\in T$. 
Since clearly $\Delta(T)\leqslant \gamma_2(T)$ we may use Corollary \ref{cor:1} for $\tau=4$ and deduce 
$$
\mathbf{E} \sup_{a\in \mathcal{A}}\sup_{t\in T}|G(a,t)|\leqslant 64C(\gamma_2(T)+M\sqrt{d}\Delta(T))\leqslant L\sqrt{d}\gamma_2(T),
$$
which by (\ref{lasek}) implies (\ref{nier:4}) - the result we refer to. 
\smallskip

\noindent
Next step is to show that the Gaussian version of the problem on Fourier series
is also related to Corollary \ref{cor:1}. In this case $U$ consists of $u\in\mathbb{C}^{I}$ of the form $u_i=\chi_i(h)$ for $h\in G$, where $\chi_i$, $i\in I$
are characters on the compact Abelian group $G$ and $T$ is any subset of $\ell^2(I)$ that contains $0$.
Recall that the crucial observation for the study is that distances $d_t$, $t\in T$ defined in (\ref{charon}) are translationally invariant on the group $G$, i.e.
$$
d_t(f\cdot h,g \cdot h)=d_t(f,g)=(\sum_{i\in I}|t_i|^2|\chi_i(f)-\chi_i(g)|^2)^{\frac{1}{2}},\;\;\mbox{for any}\;\;f,g,h\in G.
$$ 
Therefore by the deep result of Fernique (Theorem 3.1.1 in \cite{Tal1})
we have 
$$
 K^{-1}\mathcal{E}(\mathcal{A},d_t)\leqslant \mathbf{E} \sup_{a\in \mathcal{A}} |G(a,t)| \leqslant K\mathcal{E}(\mathcal{A},d_t).
$$
Consequently by Corollary \ref{cor:1} and (\ref{lasek}) we can deduce (\ref{nier:6}) which is the Gaussian version of Theorem \ref{tw:2}. 
\smallskip

\noindent
The first new consequence of Corollary \ref{cor:1} concerns shattering dimension of $U$. Suppose that $U$
is the class of real functions bounded by $1$, i.e. $U\subset [-1,1]^I$.
We say that a subset $B$ of $I$ is $\varepsilon$-shattered if there exists a level function
$v$ on $B$ such that given any subset $A$ of $B$ one can find a function $u\in U$
with $u_i\leqslant v_i$ if $i\in A$ and $u_i\geqslant v_i+\varepsilon$ if $i\in B\backslash A$.
The shattering dimension of $U$ denoted by $\mathrm{vc}(U,\varepsilon)$ is the maximal cardinality of a set $\varepsilon$-shattered by $U$. The deep result of Mendelson and Vershynin \cite{M-V} is that
for any $t\in \ell_2(I)$ such that $\|t\|_2=1$ we have
\begin{equation}\label{nier:11}
\log N(U,d_t,\varepsilon)\leqslant L\mathrm{vc}(U,c\varepsilon)\log\left(\frac{2}{\varepsilon}\right),\;\;0<\varepsilon<1, 
\end{equation}
where $L$ and $c$ are positive absolute constants. This leads to 
the following corollary.
\begin{cor}\label{cor:2}
Suppose that $0\in T\subset \ell^2(I)$ and $\sup_{u\in U}\|u\|_{\infty}$. The following inequality holds
$$
\mathbf{E} \sup_{u\in U}\sup_{t\in T}|G(u,t)|\leqslant  
K\left(g(T)+\Delta(T)\int^1_0 \sqrt{\mathrm{vc}(U,\varepsilon)\log\left(2/\varepsilon\right)}d\varepsilon\right).
$$
\end{cor}
\begin{proof}
It suffices to apply (\ref{nier:11}) then Corollary \ref{cor:1} and finally (\ref{lasek}).
\end{proof}
It is worth mentioning how to understand the shattering dimension of convex bodies.
Suppose that $U\subset [-1,1]^d$ is a convex and symmetric body then
$\mathrm{vc}(U,\varepsilon)$ is the maximal cardinality of a subset $J$ of $\{1,2,\ldots,d\}$ such that
$P_J(U)\supset [-\frac{\varepsilon}{2},\frac{\varepsilon}{2}]^J$, where $P_J$ is the orthogonal projection from $\mathbb{R}^d$ on $\mathbb{R}^J$. 
For example suppose that $U$ is a unit ball in $\mathbb{R}^d$ 
then $\mathrm{vc}(U,\varepsilon)=k-1$ for any $\varepsilon\in [\frac{2}{\sqrt{k}},\frac{2}{\sqrt{k-1}})$
and $k=1,2,\ldots d$ moreover $\mathrm{vc}(U,\varepsilon)=d$ for $\varepsilon<\frac{2}{\sqrt{d}}$. Consequently
$$
\int^1_0 \sqrt{\mathrm{vc}(U,\varepsilon)\log(2/\varepsilon)}d\varepsilon\leqslant  K\sqrt{d\log d}.
$$
Note that for $t\in \mathbb{R}^d$, $t_i=1/\sqrt{d}$ we have that $\int^{\infty}_0 \sqrt{\log N(U,d_t,\varepsilon)}d\varepsilon$ is comparable to $\sqrt{d}$
and hence the above estimate is not far from this answer.

\section{Bernoulli case}

Our next aim is to obtain a version of Corollary \ref{cor:1} in the setting of Bernoulli processes. 
We recall that by Bernoulli processes we mean
$$
X(t)=\sum_{i\in I}t_i\varepsilon_i,\;\;\mbox{for}\;\;t\in T\subset \ell^2(I),
$$
where $(\varepsilon_i)_{i\in I}$ is a Bernoulli sequence. Moreover we denoted
$$
b(T)=\mathbf{E} \sup_{t\in T}|X(t)|
$$ 
and if $0\in T$ then by Theorem \ref{tw:0} we have 
a geometrical characterization of $b(T)$. We turn to the analysis of Bernoulli processes
on product spaces, namely we consider
$$
X(u,t)=\sum_{i\in I}u_i t_i\varepsilon_i,\;\;t\in T,u\in U, 
$$
where $T\subset \mathbb{R}^I$ or $\mathbb{C}^{I}$, $U\subset \mathbb{R}^I$ or $\mathbb{C}^{I}$ and $(\varepsilon_i)_{i\in I}$ is a Bernoulli sequence. 
Our approach is to use Theorem \ref{tw:0} in order to extend Corollary \ref{cor:1} to the case of random signs. 
\begin{thm}\label{tw:5}
Suppose that $0\in T$. Then there exists $\pi:T\rightarrow \ell^2$ such that $\|\pi(t)\|_1\leqslant Lb(T)$ for all $t\in T$, $\pi(0)=0$ and
$$
\mathbf{E} \sup_{u\in U}\sup_{t\in T}|X(u,t)|\leqslant K\left(\sup_{u\in U}\|u\|_{\infty}b(T)+\sup_{t\in T}\mathcal{E}(U,d_{t-\pi(t)})\right),
$$
where
$$
\mathcal{E}(U,d_{t-\pi(t)})\sim \int^{\infty}_0 \sqrt{\log N(U,d_{t-\pi(t)},\varepsilon)}d\varepsilon.
$$
\end{thm}
\begin{proof}
Obviously we may assume that $b(T)=\mathbf{E}\sup_{t\in T}|\sum_{i\in I}t_i\varepsilon_i|<\infty$. Therefore by Theorem \ref{tw:0} there exists a decomposition $T\subset T_1+T_2$, $T_1,T_2\subset \ell_2(I)$
which satisfies $0\in T_2$ and
\begin{equation}\label{nier:12}
\max\left\{\sup_{t\in T_1}\|t\|_1,g(T_2)\right\}\leqslant Lb(T),
\end{equation}
where $L$ is a universal constant. Hence combining Corollary \ref{cor:1} with (\ref{nier:12})
\begin{align*}
& \mathbf{E}\sup_{u\in U}\sup_{t\in T}|X(u,t)|\\
&\leqslant C\left(\mathbf{E} \sup_{u \in U}\sup_{t\in T_1}|X(u,t)|+
\mathbf{E}\sup_{u\in U}\sup_{t\in T_2}|X(u,t)|\right)\\
&\leqslant C\left(\sup_{u\in U}\|u\|_{\infty}\sup_{t\in T_1}\|t\|_1+ \sqrt{\frac{\pi}{2}}\mathbf{E} \sup_{u\in U}\sup_{t\in T_2}|G(u,t)|\right)\\
&\leqslant  CL\left(\sup_{u\in U}\|u\|_{\infty}b(T)+\sup_{t\in T_2}\int^{\infty}_0 \sqrt{\log N(U,d_t,\varepsilon)}d\varepsilon\right).
\end{align*}
The decomposition of $T$ into $T_1+T_2$ can be defined in a way that $T_1=\{\pi(t):\;t\in T\}$ and
$T_2=\{t-\pi(t):\;t\in T\}$  where $\pi:T\rightarrow \ell_2$ is such that $\pi(0)=0$, $\|\pi(t)\|_1\leqslant Lb(T)$
and $\gamma_2(T_2)\leqslant Lb(T)$. It completes the proof.
\end{proof}
Our remarks on the entropy function from the previous chapter shows that Theorem \ref{tw:5} solves both our questions from the 
introduction. In fact we can easily extend our result for the functional shattering dimension.
\begin{cor}\label{cor:3}
Suppose that $\sup_{u\in U}\|u\|_{\infty}\leqslant 1$ and  $U$ is of $\varepsilon$-shattering dimension $\mathrm{vc}(\mathcal{A},\varepsilon)$ for $0<\varepsilon<1$, then
$$
\mathbf{E} \sup_{u\in U}\sup_{t\in T}|X(u,t)|\leqslant Kb(T) \int^1_0 \sqrt{\mathrm{vc}(U,\varepsilon)\log(2/\varepsilon)}d\varepsilon .
$$
\end{cor}
\begin{proof}
Obviously by Theorem \ref{tw:5},
\begin{align*}
& \mathbf{E} \sup_{u\in U}\sup_{t\in T}|X(u,t)|\\
&\leqslant K\left(\sup_{u\in U}\|u\|_{\infty}b(T)+\sup_{t\in T}\|t-\pi(t)\|_2
\int^1_0 \sqrt{\mathrm{vc}(U,\varepsilon)\log(2/\varepsilon)}d\varepsilon\right),
\end{align*}
where $\|\pi(t)\|_1\leqslant Lb(T)$. Since
$$
\sup_{t\in T}\|t-\pi(t)\|_2\leqslant Lb(T)\;\;\mbox{and}\;\;\sup_{u\in U}\|u\|_{\infty}\leqslant 1
$$
it implies that
$$
\mathbf{E} \sup_{u\in U}\sup_{t\in T}|X(u,t)|\leqslant Kb(T)\int^1_0 \sqrt{\mathrm{vc}(U,\varepsilon)\log(2/\varepsilon)}d\varepsilon.
$$
\end{proof}
\begin{cor}\label{cor:4}
Suppose that $\sup_{u\in U}\|u\|_{\infty}\leqslant 1$ and $v_i\in F$, $i\in I$, then
$$
\mathbf{E} \sup_{u\in U}\left\|\sum_{i\in I}u_iv_i\varepsilon_i\right\|\leqslant K\mathbf{E} \left\|\sum_{i\in I}v_i\varepsilon_i\right\|\int^1_0\sqrt{\mathrm{vc}(U,\varepsilon)\log(2/\varepsilon)}d\varepsilon.
$$
\end{cor}
Note that in the same way as for the maximal inequality we may ask what is the right characterization of
$U\subset \mathbb{R}^{I}$, $\sup_{u\in U}\|u\|_{\infty}\leqslant 1$ for which the inequality
\begin{equation}\label{nier:13}
\mathbf{E} \sup_{u\in U}\left\|\sum_{i\in I}u_iX_i\right\|\leqslant K\left\|\sum_{i\in I}X_i\right\|,
\end{equation}
holds for any sequence of independent symmetric r.v.'s $X_i$, $i\in I$ which take values in a separable Banach space $F$.
With the same proof as the first part of Theorem \ref{tw:1} one can show 
that $U$ should satisfy the condition
$$
\sup_{0<\varepsilon<1}\varepsilon\sqrt{\mathrm{vc}(U,\varepsilon)}<\infty.
$$
On the other hand Corollary \ref{cor:4} implies that
$$
\int^1_0 \sqrt{\mathrm{vc}(U,\varepsilon)\log(2/\varepsilon)}d\varepsilon<\infty
$$
is a sufficient condition for the inequality (\ref{nier:13}). The exact answer for this question seems to stay at the moment beyond our reach.
\smallskip

\noindent
We revisit our toy example where $U$ is the ellipsoid
$$
U=\left\{u\in \mathbb{R}^{I}:\sum_{i\in I}\frac{|u_i|^2}{|x_i|^2}\leqslant 1\right\},
$$
where $|x_i|>0$ are given numbers. One can show the following result.
\begin{rem}\label{rem:2}
Suppose that $U$ is the usual ellipsoid. Then for any set $0\in T\subset \ell^2(I)$
$$
\mathbf{E} \sup_{u\in U}\sup_{t\in T}|X(u,t)|\leqslant K\|x\|_{2}b(T).
$$
\end{rem}
\begin{proof}
We may argue in a similar way as in Remark \ref{rem:1} using this time
Theorem \ref{tw:5} and obtain
$$
\mathbf{E} \sup_{u\in U}\sup_{t\in T}|X(u,t)|\leqslant K\left(\|x\|_{\infty}b(T)+\sup_{t\in T}\|t-\pi(t)\|_2 \|x\|_{2}\right).
$$
Since $\Delta(T)\leqslant Lb(T)$ and $\|\pi(t)\|_1\leqslant Lb(T)$ we get $\sup_{t\in T}\|t-\pi(t)\|_2\leqslant Lb(T)$
and therefore
$$
\mathbf{E} \sup_{u\in U}\sup_{t\in T}|X(u,t)|\leqslant K\|x\|_2b(T).
$$
\end{proof}
On the other hand by the Schwartz inequality 
$$
\mathbf{E} \sup_{u\in U}\sup_{t\in T}|X(u,t)|=\sup_{t\in T}\mathbf{E} \sup_{u\in U}|X(u,t)|=\sup_{t\in T}\|t x\|_2
$$
Consequently in this case the expectation of the supremum of the Bernoulli process over the product of index sets 
can be explained by the analysis of one of its marginal processes.
\smallskip

\noindent
It is worth mentioning that whenever there holds a comparability of moments like (\ref{nier:13}) 
one can deduce also the comparability of tails.
\begin{rem}\label{rem:3}
Suppose that $\sup_{u\in U}\|u\|_{\infty}\leqslant 1$.
If for any $X_i$, $i\in I$ symmetric independent random variables which take values in a separable Banach space $(F,\|\cdot\|)$
\begin{equation}\label{restol2}
\mathbf{E} \sup_{u\in U} \left\|\sum_{i\in I}u_i X_i\right\|\leqslant L\mathbf{E} \left\|\sum_{i\in I}X_i\right\|,
\end{equation}
then there exists an absolute constant $K$ such that
$$
\mathbf{P}\left(\sup_{u\in U} \left\|\sum_{i\in I} u_iX_i\right\|\geqslant K t\right) \leqslant K\mathbf{P}\left(\left\|\sum_{i\in I}X_i\right\|\geqslant t\right).
$$   
\end{rem}
\begin{proof}
It suffices to prove the inequality for $X_i=v_i\varepsilon_i$, where $v_i\in F$ and $\varepsilon_i$
are independent Bernoulli variables. The general result follows then from the Fubini theorem.
By the result of Dilworth and Montgomery-Smith \cite{D-M} for all $p\geqslant 1$
\begin{align*}
& \left\|\sup_{u\in U}\left\|\sum_{i\in I}u_iv_i\varepsilon_i\right\|\right\|_p\\
&\leqslant C\left(\left\|\sup_{u\in U}\left\|\sum_{i\in I}u_iv_i\varepsilon_i\right\|\right\|_1+
\sup_{u\in U}\sup_{\|x^{\ast}\|\leqslant 1}\left\|\sum_{i\in I}u_ix^{\ast}_i(v_i)\varepsilon_i\right\|_p \right).
\end{align*}
Therefore by the Markov inequality and the assumption $\sup_{u\in U}\|u\|_{\infty}\leqslant1$
we obtain for all $p\geqslant 1$
\begin{align}
\nonumber&\mathbf{P}\left(\sup_{u\in U}\left\|\sum_{i\in I}u_iv_i\varepsilon_i\right\|\geqslant C
\left(\mathbf{E}\left\|\sum_{i\in I}v_i\varepsilon_i\right\|+\sup_{\|x^{\ast}\|\leqslant 1}\left\|\sum_{i\in I}x^{\ast}(v_i)\varepsilon_i\right\|_p\right)\right)\\
\label{restol3}&\leqslant e^{-p}.
\end{align}
On the other hand it is known (e.g. \cite{Lat3}) that for any functional $x^{\ast}\in F^{\ast}$
\begin{equation}\label{restol4}
\mathbf{P}\left(\left|\sum_{i\in I}x^{\ast}(v_i)\varepsilon_i\right|\geqslant M^{-1}\left\|\sum_{i\in I}x^{\ast}v_i\varepsilon_i\right\|_p\right)\geqslant \min\left\{c,e^{-p}\right\},
\end{equation}
where $M\geqslant 1$. Hence 
\begin{align*}
& \mathbf{P}\left(\sup_{u\in U}\left\|\sum_{i\in I}u_iv_i\varepsilon_i\right\|\geqslant C
\left(\mathbf{E}\left\|\sum_{i\in I}v_i\varepsilon_i\right\|+Mt\right)\right)\\
&\leqslant  \sup_{\|x^{\ast}\|\leqslant 1}
c^{-1} \mathbf{P}\left(\left|\sum_{i\in I}x^{\ast}(v_i)\varepsilon_i\right|>t\right)\leqslant c^{-1}\mathbf{P}\left(\left\|\sum_{i\in I}v_i\varepsilon_i\right\|>t\right).
\end{align*}
We end the proof considering two cases. If $t\geqslant M^{-1}\mathbf{E}\left\|\sum_{i\in I}v_i\varepsilon_i\right\|$ then
\begin{align*}
& \mathbf{P}\left(\sup_{u\in U}\left\|\sum_{i\in I}u_iv_i\varepsilon_i\right\|\geqslant 2CMt\right) \leqslant  c^{-1} \mathbf{P}\left(\|\sum_{i\in I}v_i\varepsilon_i\|> t\right).
\end{align*}
If $t\leqslant M^{-1}\mathbf{E}\|\sum_{i\in I}v_i\varepsilon_i\|$ then by Paley-Zygmund inequality
and the Kahane inequality with the optimal constant \cite{L-O}
$$
\mathbf{P}\left(\left\|\sum_{i\in I}v_i\varepsilon_i\right\|>t\right)\geqslant 
M^{-2}\frac{\left(\mathbf{E}\left\|\sum_{i\in I}\varepsilon_iv_i\right\|\right)^2}{\mathbf{E}\left\|\sum_{i\in I}\varepsilon_iv_i\right\|^2}\geqslant \frac{1}{2}M^{-2}.
$$
It shows that
$$
 \mathbf{P}\left(\sup_{u\in U}\left\|\sum_{i\in I}u_iv_i\varepsilon_i\right\|\geqslant K t\right)\leqslant K\mathbf{P}\left(\left\|\sum_{i\in I}v_i\varepsilon_i\right\|>t\right),
$$
which completes the proof.
\end{proof}

\section{Applications}

First application we show concerns processes on $\mathcal{A}=[0,1]$ but with values in $\mathbb{R}^{I}$. In
order to discuss whether or not $X(a)$, $a\in \mathcal{A}$ has its paths in $\ell^2(I)$ space
we have to verify the condition $\|X(a)\|_2<\infty$ a.s. for all $a\in \mathcal{A}$. This is the same
question as
$$
\sup_{\|x^{\ast}\|\leqslant 1}\sup_{a\in \mathcal{A}}\langle x^{\ast},X(a)\rangle<\infty.
$$
Hence it suffices to check the condition  $\mathbf{E} \sup_{a\in \mathcal{A}}\sup_{t\in T}|X(a,t)|<\infty$, where
$T=\{t\in\ell^2(I):\;\|t\|_2\leqslant 1\}$ and
$$
X(a,t)=\langle t,X(a)\rangle,\;\;t\in T,a\in \mathcal{A}.
$$
Note that in this way we match each random vector $X(a)$ with the process $X(a,t)$, $t\in T$.
By Theorem \ref{tw:4} (note that $X(0,s)=0$)
$$
\mathbf{E} \sup_{a\in \mathcal{A}}\sup_{t\in T}|X(a,t)|\leqslant 32(\gamma^{\tau}_{X,\mathcal{A}}(T)+\mathcal{E}^{\tau}_{X,T}(\mathcal{A})).
$$ 
In the described setting we usually expect that
\begin{equation}\label{efka1}
q_{n,a}(s,t)\leqslant q_{n,1}(s,t),\;\;a\in \mathcal{A}
\end{equation}
which means that on average the increments increases with time.
Condition (\ref{efka1}) yields $\gamma^{\tau}_{X,\mathcal{A}}(T)\leqslant \gamma^{\tau}_{X(1)}(T)$.
Under certain assumptions we may also expect that $\gamma^{\tau}_{X(1)}(T)$
is equivalent to $\mathbf{E} \|X(1)\|_2$. For example this is the case when $X(1)$ is a centred Gaussian
vector and also if $X(1)$ consists of entries  $X(1)_i$, $i\in I$
that are independent random variables that satisfy some technical assumptions as stated
in \cite{Lat1}. Moreover if there exists increasing family of functions
$\eta_n:\mathbb{R}_{+}\rightarrow \mathbb{R}_{+}$ which are continuous, increasing, $\eta_n(0)=0$
such that for all $t\in T$ and $a,b\in\mathcal{A}$
\begin{equation}\label{efka2}
q_{n,t}(a,b)\leqslant \eta_n(|a-b|)\;\;\mbox{and}\;\;\sum^{\infty}_{n=0}\eta^{-1}_{n+\tau}(N_n)<\infty
\end{equation}
then
$$
\mathcal{E}_{X,T}^{\tau}(\mathcal{A})\leqslant \sum^{\infty}_{n=0} \eta^{-1}_{n+\tau}(N_{n})<\infty.
$$ 
In this way we obtain the following remark which is a generalization results from \cite{Fer1,Fer}.
\begin{rem}\label{rem:4}
Suppose that $X(a,t)$, $a\in \mathcal{A}$, $t\in T$ satisfies (\ref{efka1}), (\ref{efka2}) and $\mathbf{E}\|X(1)\|_2$
is comparable with $\gamma_{X(1)}(T)$ then
$$
\mathbf{E}\sup_{a\in\mathcal{A}}\sup_{t\in T}|X(a,t)|\leqslant K(\mathbf{E}\|X(1)\|_2+1).
$$
\end{rem}
The second application we would like to discuss concerns empirical processes. Let $(\mathcal{X},\mathcal{B})$ be a measurable space, $\mathcal{F}$ be a countable family of measurable functions $f:\mathcal{X}\rightarrow \mathbb{R}$ such that $0\in \mathcal{F}$and $X_1,X_2,\ldots,X_N$ be family of independent random variables that satisfy
\begin{equation}\label{efka7}
\mathbf{P}\left(|(f-g)(X_i)|>d_1(f,g)t+d_2(f,g)t^{\frac{1}{2}}\right)\leqslant 2\exp(-t),\;\;\mbox{for all}\;t>0.
\end{equation}
We will analyse the case when $d_2(f,g)\geqslant d_1(f,g)$. By Exercise 9.3.5 in \cite{Tal1} for any centred independent random variables $Y_1,Y_2,\ldots,Y_n$ which satisfy
$$
\mathbf{P}\left(|Y_i|\geqslant At+Bt^{\frac{1}{2}}\right)\leqslant 2\exp(-t)\;\;\mbox{for all}\;i=1,2,\ldots,N,
$$
where $A\leqslant B$ and for any numbers $u_1,u_2,\ldots,u_N$ we have 
\begin{equation}\label{efka3}
\mathbf{P}\left(\left|\sum^N_{i=1}u_iY_i\right|\geqslant  L\left(A\|u\|_{\infty}t+B\|u\|_2 t^{\frac{1}{2}}\right)\right)\leqslant 2\exp(-t)\;\;\mbox{for all}\;
t > 1.
\end{equation}
Observe that if we define    
$$
X(u,f)=\frac{1}{\sqrt{N}}\sum^N_{i=1}\varepsilon_i u_if(X_i),\;\;u\in U,f\in \mathcal{F},
$$
where $U$ is a unit ball in $\mathbb{R}^N$ and $(\varepsilon_i)^N_{i=1}$ is a Bernoulli sequence independent of $X_i$, $i=1,2,\ldots,N$ then
$$
\mathbf{E}\sup_{u\in U}\sup_{f\in \mathcal{F}}|X(u,f)|=\mathbf{E} \sup_{f\in \mathcal{F}}\left(\frac{1}{N}\sum^N_{i=1}|f(X_i)|^2\right)^{\frac{1}{2}}.
$$
Clearly by (\ref{efka3}) used to $Y_i=\varepsilon_i(f-g)(X_i)$ we get for all $t>1$
\begin{align}
&  \nonumber \mathbf{P}\left(|X(u,f)-X(u,g)|\geqslant \frac{L}{\sqrt{N}}\left(d_1(f,g)\|u\|_{\infty}t+d_2(f,g)\|u\|_2t^{\frac{1}{2}}\right)\right)\\
&\label{efka5}\leqslant 2\exp(-t).
\end{align}
Thus in particular we can use
$$
q_{n,u}(f,g)= L(d_1(f,g)\|u\|_{\infty}2^{n}+d_2(f,g)\|u\|_2 2^{\frac{n}{2}}).
$$
Then
\begin{equation}\label{efka4}
q_{n,U}(f,g)=\frac{L}{\sqrt{N}}(d_1(f,g)2^{n}+d_2(f,g)2^{\frac{n}{2}}).
\end{equation}
Let
$$
\gamma_i(\mathcal{F},d_i)=\inf_{\mathcal{A}}\sup_{f\in \mathcal{F}}\sum^{\infty}_{n=0}2^{\frac{n}{i}}\Delta_i(A_n(f)),
$$
where the infimum is taken over all admissible partitions $\mathcal{A}=(\mathcal{A}_n)_{n\geqslant 0}$ and $\Delta_i(A)=\sup_{f,g\in A}d_i(f,g)$. 
It is easy to construct an admissible partitions which works  for both $\gamma_1(\mathcal{F},d_1)$ and $\gamma_2(\mathcal{F},d_2)$.
Namely we consider partitions $\mathcal{A}^1=(\mathcal{A}_n^1)_{n\geqslant 0}$ and $\mathcal{A}^2=(\mathcal{A}^2_n)_{n\geqslant 0}$ such that
$$
(1+\varepsilon)\gamma_2(\mathcal{F},d_i)\geqslant  \sup_{f\in \mathcal{F}}\sum^{\infty}_{n=0}2^{\frac{n}{i}}\Delta_i(A^i_n(f)),\;\;i=1,2,
$$
for some arbitrary small $\varepsilon>0$ and then define $\mathcal{A}=(\mathcal{A}_n)_{n\geqslant 0}$ by $\mathcal{A}_{n}=\mathcal{A}^1_{n-1}\cap \mathcal{A}^2_{n-1}$ for $n\geqslant 1$
and $\mathcal{A}_0=\{\mathcal{F}\}$. Obviously $\mathcal{A}$ is admissible, moreover,
\begin{equation}\label{efka8}
\sup_{f\in \mathcal{F}}\sum^{\infty}_{n=0}2^{\frac{n}{i}}\Delta_i(A_n(f)) \leqslant (1+\varepsilon)2^{\frac{1}{i}}\gamma_i(\mathcal{F},d_i),\;\;i=1,2.
\end{equation}
Using the partition $\mathcal{A}$ we derive from (\ref{efka4})
$$
\gamma^{\tau}_{X,\mathcal{A}}(\mathcal{F})\leqslant \frac{2^{\tau+1}(1+\varepsilon)}{\sqrt{N}}(\gamma_1(\mathcal{F},d_1)+\gamma_2(\mathcal{F},d_2)).
$$
On the other hand using that $X(u,f)-X(v,f)=X(u-v,f)-X(0,f)$ similarly to (\ref{efka5}) we get for all $t>1$
\begin{align*}
&\mathbf{P}\left(|X(u,f)-X(v,f)|\geqslant \frac{L}{\sqrt{N}}\left(d_1(f,0)\|u-v\|_{\infty}t+d_2(f,0)\|u-v\|_2 t^{\frac{1}{2}}\right)\right)\\
&\leqslant 2\exp(-t).
\end{align*}
Hence we may define
\begin{equation}\label{yargi}
q_{n,f}(u,v)=\frac{L}{\sqrt{N}}\left(d_1(f,0)\|u-v\|_{\infty}2^{n}+d_2(f,0)\|u-v\|_2 2^{\frac{n}{2}}\right).
\end{equation}
We aim to compute entropy numbers $e_{n,f}$, $n\geqslant 0$. Let $n_0\geqslant 1$ be such that $2^{n_0}\geqslant N\geqslant 2^{n_0-1}$.
Our first claim is that for any $n\geqslant n_0$ and suitably chosen $\sigma\geqslant 1$ it is possible to cover $B_N(0,1)$, a unit ball in $\mathbb{R}^N$, by at most $N_{n+\sigma}$ cubes of $\ell^{\infty}$
diameter at most $N^{-\frac{1}{2}}(N_n)^{-\frac{1}{N}}$. Indeed we can apply the volume type argument.
It is possible to cover $B_N(0,1)$ with $M$ disjoint cubes of $\ell^{\infty}$ diameter $t$ with disjoint interiors in $B_N(0,1+t\sqrt{N})$. Since $|B_N(0,1)|\sim (C/\sqrt{N})^N$ we get
$$
M\leqslant \frac{|B_N(0,1+t\sqrt{N})|}{t^N}\leqslant \left(\frac{C}{\sqrt{N}}\left(\frac{1+t\sqrt{N}}{t}\right)\right)^N\leqslant C^N\left(1+\frac{1}{t\sqrt{N}}\right)^N.
$$
We may choose $t=N^{-\frac{1}{2}}N_n^{-\frac{1}{N}}$, then 
$$
C^N\left(1+\frac{1}{t\sqrt{N}}\right)^N= C^N\left(1+N_n^{\frac{1}{N}}\right)^N\leqslant (N_{n+\sigma})^{\frac{1}{N}},
$$
where the last inequality uses the assumption that $n\geqslant n_0$.
In this way we covered $B_N(0,1)$ by at most $N_{n+\sigma}$ sets of $\ell_{\infty}$ diameter
at most $N^{-\frac{1}{2}}(N_n)^{-\frac{1}{N}}$ and $\ell_2$ diameter at most $N_n^{-\frac{1}{N}}$.
By (\ref{yargi}) we infer the following entropy bound
$$
e^{\tau}_{n+\sigma,f}\leqslant 
L2^{\tau}\left(d_1(f,0)N^{-1}(N_n)^{-\frac{1}{N}}+d_2(f,0)N^{-\frac{1}{2}}(N_n)^{-\frac{1}{N}}\right).
$$
We recall that constant $L$ is absolute but may change its value from line to line up to a numerical factor.  
This implies the bound 
\begin{align*}
& \sum^{\infty}_{n=n_0} e^{\tau}_{n+\sigma,f}\\
&\leqslant L2^{\tau}\left(d_1(f,0)N^{-1}(N_n)^{-\frac{1}{N}}+d_2(f,0)N^{-\frac{1}{2}}(N_n)^{-\frac{1}{N}}\right)\\
&\leqslant L2^{\tau}\left(d_1(f,0)+d_2(f,0)\right).	
\end{align*}
The second step is to consider $n\leqslant n_0/2 +\sigma$.
In this case we can simply use the trivial covering of $B_N(0,1)$
by a single set which obviously has $\ell^{\infty}$ and $\ell^2$ diameter equal $2$ and hence
$$
e^{\tau}_{n,f}=\frac{L2^{\tau}}{\sqrt{N}}(d_1(f,0)2^{n}+d_2(f,0)2^{\frac{n}{2}})
$$
and 
$$
\sum^{n_0/2+\sigma}_{n=0}e^{\tau}_{n,f}\leqslant L2^{\tau}\left(d_1(f,0)+d_2(f,0)\right).
$$
The most difficult case is when $n_0/2+\sigma\leqslant n\leqslant n_0+\sigma$.
In this setting we will cover $B_N(0,1)$ with cubes of $\ell^{\infty}$ diameter $2t$, where 
on $t=\frac{1}{\sqrt{m_n}}$ and $m_n\leqslant N$. We will not control $\ell^2$ diameter, we simply
use that it is always bounded by $2$. Note that if $x\in B_N(0,1)$
then only on at most $m_n$ coordinates such that $|x_i|>t$. Therefore we can cover $B_N(0,1)$
with cubes in $\mathbb{R}^N$ of $\ell^{\infty}$ diameter $2t$ if for each subset $J\subset \{1,2,\ldots,N\}$ 
such that $|J|=m_n$ we cover $B_J(0,1)\subset \mathbb{R}^J$ with
cubes in $\mathbb{R}^J$ of $\ell^{\infty}$ diameter $2t$. In this way we cover all the cases where 
all coordinates but those in $J$ stay in the cube $[-t,t]^{J^c}$. By our volume argument one needs at most
$$
M_J\leqslant C^J\left(1+\frac{2t}{\sqrt{|J|}}\right)^{|J|}=C^{m_n}\left(1+\frac{1}{2t\sqrt{m_n}}\right)^{m_n}
$$
$2t$-cubes in $\mathbb{R}^J$ to cover $B_J(0,1)$. Our choice of $t$ guarantees that
$M_J\leqslant (3c/2)^{m_n}$. Therefore one needs $\binom{N}{m_n} (3c/2)^{m_n}$ cubes of $\ell^{\infty}$ diameter $2t$
to cover $B_N(0,1)$. We require that
$$
\binom{N}{m_n} \left(\frac{3c}{2}\right)^{m_n}\leqslant N_{n+\sigma}.
$$
It remains to find $m_n$ that satisfies the above inequality. First observe that
$$
\binom{N}{m_n} \leqslant \left(\frac{eN}{m_n}\right)^{m_n}=\exp(m_n\log(eN/m_n)).
$$
Following Talagrand we define $m_n$ as the smallest integer such that
$$
2^{n-\sigma}\leqslant m_n\log(eN/m_n).
$$
Clearly if $n\geqslant n_0/2+\sigma$ then $m_n>1$ and thus
by Lemma 9.3.12 in \cite{Tal1} we deduce that $m_n\log(eN/m_n)\leqslant 2^{n-\sigma+1}$ and hence
$$
\binom{N}{m_n} \left(\frac{3c}{2}\right)^{m_n}\leqslant 
\exp(2^{n-\sigma+1})\left(\frac{3c}{2}\right)^{m_n}\leqslant N_{n+\sigma},
$$
for sufficiently large $\sigma$. Again by Lemma 9.3.12 in \cite{Tal1} we have $\frac{1}{8}m_{n+1}\leqslant m_n$ for all $n\in \{n_0/2+\sigma,\ldots n_0+\sigma-1\}$ and $m_{n_0+\sigma}=2^{n_0}\geqslant N$. It implies that
\begin{equation}\label{nuwa}
m_n\geqslant N\left(\frac{1}{8}\right)^{n_0+\sigma-n}.
\end{equation}
Recall that each of the covering cubes has $\ell^2$
diameter $2$ and therefore by the definition of $m_n$ and then by (\ref{nuwa})
\begin{align*}
& e^{\tau}_{n,f}\leqslant \frac{L2^{\tau}}{\sqrt{N}}\left(d_1(f,0)\frac{2^{n}}{\sqrt{m_n}}+d_2(f,0)2^{\frac{n}{2}}\right)
 \\
&\leqslant L 2^{\tau+\frac{n-n_0}{2}}\left(d_1(f,0)\sqrt{\log(eN/m_n)}+d_2(f,0)\right)\\
&\leqslant L2^{\tau+\frac{n-n_0}{2}}(d_1(f,0)\sqrt{1+(n_0+\sigma-n)\log(8)}+d_2(f,0)).
\end{align*}
Again we derive the bound
$$
\sum^{n_0+\sigma}_{n=n_0/2+\sigma} e^{\tau}_{n,f}\leqslant L2^{\tau}(d_1(f,0)+d_2(f,0)).
$$
We have established that
$$
\mathcal{E}^{\tau}_{X,f}(U)=\sum^{\infty}_{n=0}e^{\tau}_{n,f}\leqslant L2^{\tau}(d_1(f,0)+d_2(f,0))
$$
and consequently
$$
\mathcal{E}^{\tau}_{X,\mathcal{F}}(U)=\sup_{f\in \mathcal{F}}\mathcal{E}^{\tau}_{X,f}(U)\leqslant L2^{\tau}\sup_{f\in \mathcal{F}}(d_1(f,0)+d_2(f,0)).
$$
By Theorem \ref{tw:4} we get
\begin{align*}
& \mathbf{E}\sup_{f\in \mathcal{F}}(\frac{1}{\sqrt{N}}
\sum^N_{i=1}|f(X_i)|^2)^{\frac{1}{2}}\\
& \leqslant K\left(\frac{1}{\sqrt{N}}(\gamma_1(\mathcal{F},d_1)+\gamma_2(\mathcal{F},d_2))+\Delta_1(\mathcal{F})+\Delta_2(\mathcal{F})\right).
\end{align*}
Note that our assumption that $d_2$ dominates $d_1$ is not necessary for the result. Clearly using $\bar{d}_2=\max(d_1,d_2)$ instead of $d_2$ we only have to observe that for our admissible partition $\mathcal{A}$ which works for $\gamma_1(\mathcal{F},d_1)$ and $\gamma_2(\mathcal{F},d_2)$ in the sense of (\ref{efka8})
one can use the following inequality
$$
\sum^{\infty}_{n=0}2^{\frac{n}{2}}\bar{\Delta}_2(A_n(f))\leqslant \sum^{\infty}_{n=0}2^n\Delta_1(A_n(f))+\sum^{\infty}_{n=0}2^{\frac{n}{2}}\Delta_2(A_n(f)),
$$
where $\bar{\Delta}_2(A)$ is the diameter of $A$ with respect to $\bar{d}_2$ distance.
In the same way $\bar{\Delta}_2(\mathcal{F})\leqslant \Delta_1(\mathcal{F})+\Delta_2(\mathcal{F})$.
We have proved the following result.
\begin{thm}\label{tw:6}
Suppose that $0\in \mathcal{F}$ and $\mathcal{F}$ satisfies (\ref{efka7}). Then
\begin{align*}
&\mathbf{E}\sup_{f\in \mathcal{F}}\frac{1}{\sqrt{N}}\left(\sum^N_{i=1}|f(X_i)|^2\right)^{\frac{1}{2}}\\
&\leqslant K\left(\frac{1}{\sqrt{N}}\left(\gamma_1(\mathcal{F},d_1)+\gamma_2(\mathcal{F},d_2)\right)+\Delta_1(\mathcal{F})+\Delta_2(\mathcal{F})\right).
\end{align*}
\end{thm}
The result is due to Mendelson and Paouris \cite{Me-P} (see Theorem 9.3.1 in \cite{Tal1}) and concerns slightly more general situation. The proof we have shown is different and much less technically involved.

\end{document}